\documentclass[a4paper]{article}
\usepackage[affil-it]{authblk}
\usepackage[utf8]{inputenc}
\usepackage{natbib}
\setcitestyle{authoryear}
\usepackage{xcolor}
\usepackage{amsmath}
\usepackage{amsfonts}
\usepackage{amssymb}
\usepackage{amsthm}
\usepackage{stmaryrd}
\usepackage{hyperref}
\usepackage{tikz}
\usetikzlibrary{cd,arrows.meta,decorations.markings}

\tikzset{
    nattrans/.style={
        -Implies,
        double distance=2pt
    },
    onattrans/.style={
        decoration={
            markings,
            mark=at position 0.5 with {\draw (0,0) circle (0.2em);}},
        -Implies,
        double distance=2pt,
        postaction={decorate}  
    },
}
    
\makeatletter
\newcommand{\rightarrowcircle}{}
\DeclareRobustCommand{\rightarrowcircle}{%
  \mathrel{\vphantom{\rightarrow}\mathpalette\circle@arrow\relax}%
}
\newcommand{\circle@arrow}[2]{%
  \m@th
  \ooalign{%
    \hidewidth$#1\circ\mkern1mu$\hidewidth\cr
    $#1\longrightarrow$\cr}%
}
\makeatother

\makeatletter
\newcommand{\Rightarrowcircle}{}
\DeclareRobustCommand{\Rightarrowcircle}{%
  \mathrel{\vphantom{\Rightarrow}\mathpalette\Circle@arrow\relax}%
}
\newcommand{\Circle@arrow}[2]{%
  \m@th
  \ooalign{%
    \hidewidth$#1\circ\mkern1mu$\hidewidth\cr
    $#1\Longrightarrow$\cr}%
}
\makeatother

\title{The~Bicategory~of~Open~Functors}
\author{Alexandre Fernandez}
\author{Luidnel Maignan}
\author{Antoine Spicher}
\affil{Univ Paris Est Creteil, LACL, 94000, Creteil, France}

\newtheorem{definition}{Definition}
\newtheorem{proposition}[definition]{Proposition}
\theoremstyle{remark}
\newtheorem{remark}[definition]{Remark}

\begin{document}

\newcommand\Set{\ensuremath{\textsc{Set}}}
\newcommand\Cat{\ensuremath{\textsc{Cat}}}
\newcommand\CAT{\ensuremath{\textsc{CAT}}}
\newcommand\OCat{\ensuremath{\textsc{OCat}}}
\newcommand\A{\ensuremath{\mathcal{A}}}
\newcommand\B{\ensuremath{\mathcal{B}}}
\newcommand\C{\ensuremath{\mathcal{C}}}
\newcommand\D{\ensuremath{\mathcal{D}}}
\newcommand\E{\ensuremath{\mathcal{E}}}
\newcommand\na{\ensuremath{\mathsf{a}}}
\newcommand\nl{\ensuremath{\mathsf{lu}}}
\newcommand\nr{\ensuremath{\mathsf{ru}}}
\renewcommand\ni{\ensuremath{\mathsf{i}}}

\newcommand\op{\ensuremath{\mathrm{op}}}
\newcommand\po{\ensuremath{\mathrm{po}}}
\newcommand\Arr{\ensuremath{\mathrm{Arr}}}
\newcommand\OFun[2]{\ensuremath{\llbracket#1,#2\rrbracket}}
\newcommand\ONat[2]{\ensuremath{\llbracket#1,#2\rrbracket}}
\newcommand\Alt[1]{\ensuremath{{#1}_\alpha}}
\newcommand\Res[1]{\ensuremath{{#1}_\beta}}
\newcommand\alt[1]{\ensuremath{{#1}_\gamma}}
\newcommand\res[1]{\ensuremath{{#1}_\delta}}
\newcommand\El[1]{\ensuremath{{\smallint}{#1}}}
\newcommand\id[1]{\ensuremath{\mathrm{id}[{#1}]}}
\newcommand\Id[1]{\ensuremath{\mathrm{Id}[{#1}]}}
\newcommand\ID[1]{\ensuremath{\iota[{#1}]}}
\newcommand\Open{\ensuremath{\mathbf{O}}}
\newcommand\free[1]{\ensuremath{\overline{#1}}}
\newcommand\opto{\ensuremath{\,\begin{tikz}\draw[-<] (0,0) -- (1em,0);\end{tikz}\,}}
\newcommand\oto{\ensuremath{\rightarrowcircle}}
\newcommand\oTo{\ensuremath{\Rightarrowcircle}}
\newcommand\To{\ensuremath{\Longrightarrow}}

\newcommand\eg{\emph{e.g.} }
\newcommand\ie{\emph{i.e.} }

\newcommand\mnote[1]{\textcolor{red}{#1}}

\maketitle

\begin{abstract}
    We want to replace categories, functors and natural transformations by categories, \emph{open} functors and \emph{open} natural transformations.
    In analogy with \emph{open} dynamical systems, the adjective \emph{open} is added here to mean that some external information is taken into account.
    For the particular use of the authors, such an open functor is described by two components: a presheaf representing the possible external influences for each input, and a classical functor from the category of elements of this presheaf to the category of results.
    Considering the appropriate notion of composition then leads to a bicategory.
    This report describes this bicategory with as little auxiliary constructions as possible and gives all the details of all the proofs needed to establish the bicategory, as explicitly as possible.
    Subsequent reports will give other presentations of this bicategory and compare it to other existing constructions, \eg spans, fibrations, pseudo-adjunctions, Kleisli bicategories of pseudo-monads, and profunctors (or distributors).
\end{abstract}

\section{Introduction}

We want to replace categories, functors and natural transformations by categories, \emph{open} functors and \emph{open} natural transformations.
In analogy with \emph{open} dynamical systems, the adjective \emph{open} is added here to mean that some external information is taken into account.

The origin of this will is that we found ourselve in a case where we needed to generalize the powerset monad $\mathrm{P} : \Set \to \Set$ in a setting where sets are replaced by categories.
More precisely, the goal was to represent a ``local definition to global definition'' phenomenon as a pointwise Kan extension as we usually do when thinking of spatially extended dynamical systems in terms of \emph{global transformations} (\cite{DBLP:conf/gg/MaignanS15,DBLP:conf/uc/FernandezMS19}), except that, this time, some non-determinism came into play.
The first tries was to replace objects by sets of objects and morphisms by some sort of sets of morphisms.
However, we found convincing examples showing that multiplicities are important, and these multiplicities come from the different ways and reasons why a result could be obtained.
These ``reasons'' form the aforementioned external information,
not present in the input, but nonetheless taken into account to produce the result.
This external information can be the result of a choice, either from a random source, from an external scheduler, or simply from another part of the system that is not modeled but just taken as secondary input instead (\cite{DBLP:conf/saso/MaignanG08, DBLP:journals/ppl/MaignanG09,DBLP:conf/acri/MaignanG10,DBLP:journals/fuin/ArrighiST13}).

Although it seemed clear that some settings should already exist to model this, there were questions on small details in each possible setting.
In such a case, it is a good idea to first formalize the intuition as directly as possible, and only then to try to map it in other settings in a precise formal way to fix those details.
This is exactly the purpose of this work.
In this report, we only present open functors and open natural transformations as directly as possible, with as little reference to other auxiliary constructions as possible.
Proofs are also made at a very elementary and explicit level for various reasons, including the checking of some details, the use of some parts as exercises, and the accessibility for some intented audience.
Also we do not spend time on any foundational issues about sets of sets or categories of categories.
The rest of the work should come as additional reports hopefully, and we will try to update this report to link the different reports together.
A short description of the remaining content is given at the end.

\section{Background and Notations}

Let us compile here all the categorical definitions manipulated in this report, in order to make the proofs easier to check and to fix the notations.
Note that we did our best to indicate every time a definition or proposition of the report is used.
So if a step of proof as no indication, its justification should be in this section.
For more information, one can consult \cite{maclane2013categories,borceux2008handbook,Benabou1967bicat}.

\subsection{The Strict 2-Category of Categories}

We denote by $\Cat$ the strict 2-category of categories, functors and natural transformations.

By default, the composition of two arrows $f : c \to c'$ and $g : c' \to c''$ in a category $\C$ is denoted $g \cdot f : c \to c''$, and the identity arrow of an object $c \in \C$ is denoted $\id{c} : c \to c$.
The collection of arrows from an object $c$ to an object $c'$ is denoted $[c,c']$.
In these notations, the category $\C$ is left implicit but can be retrieve from the context.

The application of a functor $F$ on an object $c$ and an arrow $f$ are denoted $F(c)$ and $F(f)$ respectively.
Also, $\Id{\C} : \C \to \C$ is the identity functor of the category $\C$ and $G \circ F : \C \to \E$ the composition of two functors $F : \C \to \D$ and $G : \D \to \E$.

The component at $c \in \C$ of a natural transformation $\theta : F \To G : \C \to \D$ is denoted $\theta[c] : F(c) \to G(c)$.
We also use the ``arrow component'' notation, \ie for any $f : c \to c'$, $\theta[f] : F(c) \to G(c')$ is defined as $\theta[f] = G(f) \cdot \theta[c] = \theta[c'] \cdot F(f)$.
Also, $\ID{F} : F \To F : \C \to \D$ is the identity natural transformation of a functor $F : \C \to \D$, $\phi \bullet \theta : F \to H : \C \to \D$ the vertical composition of $\theta : F \To G : \C \to \D$ and $\phi : G \To H : \C \to \D$, and $\phi \circ \theta : G \circ F \to G' \circ F' : \C \to \E$ the horizontal composition of $\theta : F \To F' : \C \to \D$ and $\phi : G \To G' : \D \to \E$.
With our notations, the components of these compositions are $(\phi \bullet \theta)[c] = \phi[c] \cdot \theta[c]$ and $(\phi \circ \theta)[c] = \phi[\theta[c]]$ for any $c \in \C$.
Also $(\phi \bullet \theta)[g\cdot f] = \phi[g] \cdot \theta[f]$ and $(\phi \circ \theta)[f] = \phi[\theta[f]]$ for any $f : c \to c' \in \C$.
We finally recall the exchange law $(\phi' \bullet \phi) \circ (\theta' \bullet \theta) = (\phi' \circ \theta') \bullet (\phi \circ \theta)$ and its special case $(\phi' \bullet \phi) \circ \ID{F} = (\phi' \circ \ID{F}) \bullet (\phi \circ \ID{F})$.
Recall that the exchange law only expresses that composition of functors is functorial.

Note that we use brackets where the literature typically uses subscripts because we need lots of nestings, \eg $\psi[\phi[\theta[c]]]$ instead of $\psi_{\phi_{\theta_c}}$ (and this is a simple instance).
Also, the ``arrow component'' notation is used as it is practical not to choose between equal compositions, \eg for some $\theta : F \To F'$, $\phi : G \To G'$, $\psi : H \To H'$, we write $\psi[\phi[\theta[c]]]$ instead of $\psi[G'(F'(c))] \cdot H(\phi[F'(c)]) \cdot H(G(\theta[c]))$ or $H'(\phi[F'(c)]) \cdot \psi[G(F'(c))] \cdot H(G(\theta[c]))$ or $\psi[G'(F'(c))] \cdot H(G'(\theta[c])) \cdot H(\phi[F(c)])$ or $H'(G'(\theta[c])) \cdot \phi[G'(F(c))] \cdot H(\phi[F(c)])$ or $H'(\phi[F'(c)]) \cdot H'(G(\theta[c])) \cdot \psi[G(F(c))]$ or $H'(G'(\theta[c])) \cdot H'(\phi[F(c)]) \cdot \psi[G(F(c))]$.
Diagrammatically, this particular example amounts to consider the naturality cube induced by the three natural transformations and to link directly the initial corner to the terminal corner by naming the diagonal arrow, thanks to the arrow component notation, instead of expressing it by one of the six paths along the edges.
This notation also makes some reasoning easier to write, \eg the definition of horizontal composition of natural transformations.
This lengthy paragraph is required as it prepare for the proofs below.

\subsection{The Category of Sets, its Opposite and Elements}

We denote by $\Set$ the category of sets and functions, and denote by $g \circ f$ the composition of two functions $f : X \to Y$ and $g : Y \to Z$.
Its opposite category is denoted $\Set^\op$, its objects are simple sets, but an arrow $f$ from $Y$ to $X$ in $\Set^\op$ is denoted $f : Y \opto X$ and its associated function $f^\po : X \to Y$.
This is done to reduce the risk of a disordering.
Given a functor $F : \C \to \Set^\op$, we denote by $\El{F}$ its category of elements whose objects are of the form $(c, x)$ with $c \in \C$ and $x \in F(c)$, and arrows of the form $(f, x') : (c, F(f)^\po(x')) \to (c',x')$ with $f : c \to c'$ and $x' \in F(c')$.
The composition in $\El{F}$ is $(g, x'') \cdot (f, x') := (g \cdot f, x'')$ (with $x' = F(g)^\po(x'')$ for arrows to be composable) and the identity arrows $\id{c, x} := (\id{c}, x)$.
Given another functor $G : \C \to \Set^\op$ and a natural transformation $\theta : F \To G$, we denote by $(\El{\theta})^\po : \El{G} \to \El{F}$ the functor defined as $(\El{\theta})^\po(c,y) = (c,\theta[c]^\po(y))$ and $(\El{\theta})^\po(f,y') = (f,\theta[c']^\po(y'))$ with $f : c \to c'$.
We recall also that for any fixed category $\C$, categories of elements and functors of elements form a functor $\El{\--}$ from $[\C,\Set^\op]$ to $\Cat^\op$.

\begin{center}
\begin{tikzpicture}
    \node (c)     at (0,0) {$\C$};
    \node (setop) at (3,0) {$\Set^\op$} ;

    \draw[->](c) to[bend left=+50] coordinate(f) node[above]{$F$}(setop);
    \draw[->](c) to[bend left=-50] coordinate(g) node[below]{$G$}(setop);

    \node (elf) at (4,+1) {$\El{F}$};
    \node (elg) at (4,-1) {$\El{G}$};

    \draw[nattrans,shorten <=5pt,shorten >=5pt] (f) to node[right] {$\theta$} (g);
    
    \draw[<-] (elf) to node[right]{$(\El{\theta})^\po$} (elg);

    \node (x) at (5.6,+1) {$(c,\theta[c]^\po(y))$};
    \node (xx) at (10.8,+1) {$(c',\theta[c']^\po(y'))$};
    \draw[->] (x) to node[fill=white,fill opacity=0.8,text opacity=1](xxx){$(f,\theta[c']^\po(y'))$} (xx);
    
    \node (y) at (5.6,-1) {$(c,y)$};
    \node (yy) at (10.8,-1) {$(c',y')$};
    \draw[->] (y) to node[fill=white,fill opacity=0.8,text opacity=1](yyy){$(f,y')$} (yy);

    \draw[<-|,shorten <=9pt,shorten >=9pt] (x) to node[right] {$(\El{\theta})^\po$} (y);
    \draw[<-|,shorten <=9pt,shorten >=9pt] (xx) to node[right] {$(\El{\theta})^\po$} (yy);
    \draw[<-|,shorten <=9pt,shorten >=9pt] (xxx) to node[right] {$(\El{\theta})^\po$} (yyy);

\end{tikzpicture}
\end{center}

Note that we write $\id{c, x}$ and not $\id{(c, x)}$, and we write $F(c, x)$ and not $F((c, x))$ (as done for $(\El{\theta})^\po(c,y)$ above).
This shorthand is only used for any application on objects or arrows from a category of elements.
For an arbitrary pair, we will denote $\langle x,y \rangle$ and write $F(\langle x,y \rangle)$.
To make things more readable, we rarely recall the domains and codomains of any arrow $(f, x')$.
Our experience is that these cumbersome data obscure things with too much details that are easier to retrieve than to read.
Finally, the common usage is to use $[\C^\op,\Set]$ instead of $[\C,\Set^\op]$ but we find it easier to switch convention for the particular needs of this work, so that the arrows of $\C$ are always used in the same direction during composition of open functors below, where the codomain of an open functor is also used as the domain of the following functor.

\subsection{Bicategories}

A bicategory, a notion introduced by \cite{Benabou1967bicat}, is almost like a category, except that between two objects there is not just a collection of arrows but a category of arrows.
Similarly, the composition is not simply a function but a functor.
If it is still associative, left unital and right unital for the identities, we have a strict 2-category, \eg \Cat.
But since arrows form categories, we can have different but isomorphic arrows, allowing to consider compositions that are associative, left unital, and right unital only up to isomorphisms, as long as the chosen isomorphisms are coherent.
Here, coherence means that the two ways to transform $(i \circ (h \circ (g \circ f)))$ into $(((i \circ h) \circ g) \circ f)$ should be equal, and similarly for the two transformations from $(g \circ (\id{c'} \circ f))$ to $g \circ f$.

More formally, a bicategory is composed of
(1) a collection of 0-cells,
(2) for each pair of 0-cells $c, c'$ a category $(c \to c')$ whose objects are called 1-cells, and arrows 2-cells
(3) for each 0-cell $c$ an identity 1-cell $\ni[c] : c \to c$,
(4) for each triplet of 0-cell $c, c', c''$ a composition functor $\circ_{c,c',c''}$ from $(c'\to c'') \times (c \to c')$ to $(c\to c'')$,
(5) for each pair of 0-cell $c, c'$ a natural isomorphism $\nl_{c,c'}$, called left unitor, from the functor $(\ni[c'] \circ \--)$ to the identity functor $(\--)$ both from $(c \to c')$ to $(c \to c')$
(6) for each pair of 0-cell $c, c'$ a natural isomorphism $\nr_{c,c'}$, called right unitor, from functor $(\-- \circ \ni[c])$ to identity functor $(\--)$ both from $(c \to c')$ to $(c \to c')$
(7) for each quadruplet of 0-cell $c, c', c'', c'''$ a natural isomorphism $\na_{c,c',c'',c'''}$, called associator, from functor $(\-- \circ (\-- \circ \--))$ to functor $((\-- \circ \--) \circ \--)$ both from $(c'' \to c''') \times (c' \to c'') \times (c \to c')$ to $(c \to c''')$.
These data should be coherent in the sense that this pentagon diagram in $(c\to c'''')$ and triangle diagram in $(c \to c'')$ should commute.
\begin{center}
\begin{tikzpicture}
\node (a) at (0,3) {$i \circ (h \circ (g \circ f))$};
\node (b) at (0,1.5) {$(i \circ h) \circ (g \circ f)$};
\node (c) at (8,3) {$i \circ ((h \circ g) \circ f)$};
\node (d) at (8,1.5) {$(i \circ (h \circ g)) \circ f$};
\node (e) at (4,0) {$((i \circ h) \circ g) \circ f$};
\draw[->] (a) to node[right]{$\na[g\circ f,h,i]$} (b);
\draw[->] (b) to node[below left]{$\na[f,g,i \circ h]$} (e);
\draw[->] (a) to node[above]{$\id{i}\circ \na[f,g,h]$} (c);
\draw[->] (c) to node[left]{$\na[f,h \circ g,i]$} (d);
\draw[->] (d) to node[below right]{$\na[i,h,g] \circ \id{f}$} (e);
\end{tikzpicture}
\end{center}

\begin{center}
\begin{tikzpicture}
\node (a) at (0,1.5) {$g \circ (\ni[c] \circ f)$};
\node (b) at (8,1.5) {$(g \circ \ni[c]) \circ f$};
\node (c) at (4,0) {$g \circ f$};
\draw[->] (a) to node[above]{$\na[f,\ni[c],g]$} (b);
\draw[->] (a) to node[below left]{$\id{g} \circ \nl[f]$} (c);
\draw[->] (b) to node[below right]{$\nr[g] \circ \id{f}$} (c);
\end{tikzpicture}
\end{center}
The vocabulary ``0-cell'', ``1-cell'', ``2-cell'' is only used here to make the previous paragraph readable.
Below, 0-cells are categories, 1-cells are open functors and 2-cells are open natural transformations.

\section{The Bicategory of Open Functors}

\begin{remark}
    Given two categories $\C$ and $\D$, an open functor $F$ from $\C$ to $\D$ needs to describe for each input $c\in\C$ the set of possible external interactions available for $c$.
    In all examples considered up to now by the authors, it is the case that for each $f : c \to c'$, there is an associated ``restriction'' map in the other direction, namely from the set of interactions available for $c'$ into the set of those available for $c$.
    The reason is that, in the considered examples, $c$ can be viewed as a part of $c'$ through $f$, so that each external interaction on the whole of $c'$ can be restricted to an interaction $c$ alone.
    We therefore have a functor $\Alt{F} : \C \to \Set^\op$ describing these interactions and restrictions and its category of elements $\El{\Alt{F}}$ represents the complete information, inputs and their external information/interactions.
    From this ``complete'' category, the data of a classical functor $\Res{F} : \El{\Alt{F}} \to \D$ finishes the description of the open functor $F$.
\end{remark}
\begin{definition}\label{def:open-functor}
    Given two categories $\C$ and $\D$, an \emph{open functor} $F$ \emph{from $\C$ to $\D$}, denoted $F : \C \oto \D$, is given by two functors $\Alt{F} : \C \to \Set^\op$ and $\Res{F} : \El{\Alt{F}} \to \D$, as in the following diagram.
    \begin{center}
    \begin{tikzpicture}[xscale=0.75,every node/.style={fill=white,fill opacity=0.8,text opacity=1}]
        \node (c)     at (0,0)   {$\C$};
        \node (setop) at (3,0)   {$\Set^\op$};
        \node (d)     at (7.5,0) {$\D$};

        \draw[->] (c) to (setop);
        \node (f) at (1.2,0) {$\Alt{F}$};
        \node (elf) at (4.5,0) {$\El{\Alt{F}}$};
        \draw[->] (elf) to (d);
        \node (rf) at (6,0) {$\Res{F}$};
    \end{tikzpicture}
    \end{center}
    We denote $\OFun{\C}{\D}$ the collection of all open functors from $\C$ to $\D$.
\end{definition}

\begin{remark}\label{rem:readable1}
    Given two objects $c, c' \in \C$, an arrow $f : c \to c'$, and an element $x' \in \Alt{F}(c')$, the arrow $(f, x') : (c,\Alt{F}(f)^\po(x')) \to (c',x')$ in the category of elements $\El{\Alt{F}}$ is sent to $\Res{F}(f,x') : \Res{F}(c,\Alt{F}(f)^\po(x')) \to \Res{F}(c',x')$.
    We do not recall systematically those domains and codomains below.
\end{remark}

\begin{definition}\label{def:open-identity-functor}
    The \emph{identity open functor} $\Id{\C} : \C \oto \C$ of a category $\C$ is given by $\Alt{\Id{\C}} = (c \mapsto \{ \star \}, f \mapsto \id{\{\star\}}^\op)$, and $\Res{\Id{\C}} = ((c,\star) \mapsto c,\; (f,\star) \mapsto f)$.
\end{definition}

\begin{remark}
    More generally, any classical functor can be turned into an open functor by setting no external information, \ie a singleton set for every input, as done in this last definition for the classical identity functor.
    We leave this idea for a later report.
\end{remark}

Given two open functors $F$ and $G$, their composition $G\circ F$ has a $\Alt{G\circ F}$ component corresponding of pairs of the form $\langle x, y \rangle$ with $x$ taken in $\Alt{F}$ and $y$ taken in $\Alt{G}$ in the appropriate manner so that $\Res{(G\circ F)}(c,\langle x,y \rangle) = \Res{G}(\Res{F}(c,x),y)$.
\begin{definition}\label{def:open-composition}
    Given three categories $\C$, $\D$ and $\E$ and two open functors $F : \C \oto \D$ and $G : \D \oto \E$, as in the following diagram,
    \begin{center}
    \begin{tikzpicture}[xscale=0.75,every node/.style={fill=white,fill opacity=0.8,text opacity=1}]
        \node (c)     at (0,0)   {$\C$};
        \node (setop) at (3,0)   {$\Set^\op$};
        \node (d)     at (7.5,0) {$\D$};
        \node (setop2) at (10.5,0)   {$\Set^\op$};
        \node (e)     at (15,0) {$\E$};
    
        \draw[->] (c) to (setop);
        \node (f) at (1.2,0) {$\Alt{F}$};
        \node (elf) at (4.5,0) {$\El{\Alt{F}}$};
        \draw[->] (elf) to (d);
        \node (rf) at (6,0) {$\Res{F}$};
        
        \draw[->] (d) to (setop2);
        \node (g) at (8.7,0) {$\Alt{G}$};
        \node (elg) at (12,0) {$\El{\Alt{G}}$};
        \draw[->] (elg) to (e);
        \node (rg) at (13.5,0) {$\Res{G}$};
    \end{tikzpicture}
    \end{center}
    their composition $G \circ F : \C \oto \D$ is given by:
    \begin{eqnarray*}
        &&\Alt{(G\circ F)}(c) = \coprod_{x \in \Alt{F}(c)}(\Alt{G} \circ \Res{F})(c,x), \\
        &&\Alt{(G\circ F)}(f)^\po(\langle x', y'\rangle) = \langle \Alt{F}(f)^\po(x'),(\Alt{G} \circ \Res{F})(f,x')^\po(y') \rangle, \\
        &&\Res{(G\circ F)}(c,\langle x,y \rangle) = \Res{G}(\Res{F}(c,x),y), \text{ and}\\
        &&\Res{(G\circ F)}(f,\langle x',y' \rangle) = \Res{G}(\Res{F}(f,x'),y')
    \end{eqnarray*}
    for any $c,c' \in \C$, $f : c \to c' \in \C$, $x \in \Alt{F}(c)$, $x' \in \Alt{F}(c')$ such that $x = \Alt{F}(f)^\po(x')$, $y \in (\Alt{G} \circ \Res{F})(c,x)$ and 
    $y' \in (\Alt{G} \circ \Res{F})(c,x')$ such that $y = \Alt{G}(\Res{F}(f,x'))^\po(y')$.
\end{definition}

\begin{proposition}
    All identity open functors and all open functor compositions are indeed open functors.
\end{proposition}
\begin{proof}
    Although obvious, one should still remember to check this.
\end{proof}

It is now easy to see that this composition is neither strictly associative, strictly left unital, nor strictly right unital.
Indeed, the more we compose, the more we have pairs of pairs, even when we compose with open identity functors.
In fact, the nesting of the pairs mimics the nesting of the compositions.
This is why we have a bicategory and not a 2-category.
So let us now introduce the open natural transformations first, in order to state the properties of the composition and its relations with identity open functors up to open natural isomorphisms.

\begin{remark}
    The following definitions of open natural transformations can be justified as being intuitively induced by the diagrams.
    It becomes even obvious in other presentations of open functors, but this is out of the scope of this report.
    We only propose a comparison in Remark~\ref{rem:vertical-composition} below as an additional clue.
\end{remark}

\begin{definition}\label{def:open-nat-trans}
    Given two categories $\C$ and $\D$ and two open functors $F, G : \C \oto \D$, an \emph{open natural transformation} $\theta$ \emph{from $F$ to $G$}, denoted $\theta : F \oTo G : \C \oto \D$, is given by two natural transformations $\Alt{\theta} : \Alt{F} \To \Alt{G} : \C \to \Set^\op$ and $\Res{\theta} : \Res{F}\circ(\El{\Alt{\theta}})^\po \To \Res{G} : \El{\Alt{G}} \to \D$, as in the following diagram.
    \begin{center}
    \begin{tikzpicture}[every node/.style={fill=white,fill opacity=0.8,text opacity=1}]
        \node (c)     at (0,0)   {$\C$};
        \node (setop) at (3,0)   {$\Set^\op$};
        \node (d)     at (7.5,0) {$\D$};
    
        \draw[->] (c) to[bend left=63] (setop);
        \node (f) at (1.5,+1) {$\Alt{F}$};
        \node (elf) at (4.5,+1) {$\El{\Alt{F}}$};
        \draw[->] (elf) to[bend left=22] (d);
        \node (rf) at (6,+1) {$\Res{F}$};
        
        \draw[->] (c) to[bend left=-63] (setop);
        \node (g) at (1.5,-1) {$\Alt{G}$};
        \node (elg) at (4.5,-1) {$\El{\Alt{G}}$};
        \draw[->] (elg) to[bend left=-22] (d);
        \node (rg) at (6,-1) {$\Res{G}$};
    
        \draw[nattrans] (f) to (g);
        \node at (1.5,0) {$\Alt{\theta}$};
        \draw[<-] (elf) to (elg);
        \node at (4.5,0) {$(\El{\Alt{\theta}})^\po$};
        \draw[nattrans] (rf) to (rg);
        \node at (6,0) {$\Res{\theta}$};
    \end{tikzpicture}
    \end{center}
    We denote by $\ONat{F}{G}$ the collection of all open natural transformations between them.
\end{definition}

\begin{definition}\label{def:identity-open-nat-trans}
    Given two categories $\C$ and $\D$ and an open functor $F : \C \oto \D$, its \emph{identity open natural transformation $\ID{F} : F \oTo F$} is the open natural transformation given by $\Alt{\ID{F}} = \ID{\Alt{F}}$ and $\Res{\ID{F}} = \ID{\Res{F}}$.
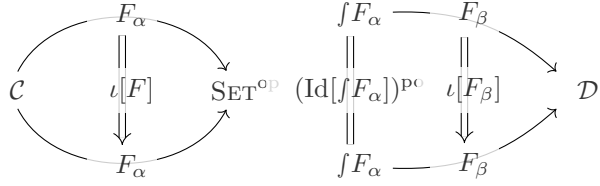
    \begin{center}
    \begin{tikzpicture}[every node/.style={fill=white,fill opacity=0.8,text opacity=1}]
        \node (c)     at (0,0)   {$\C$};
        \node (setop) at (3,0)   {$\Set^\op$};
        \node (d)     at (7.5,0) {$\D$};
    
        \draw[->] (c) to[bend left=63] (setop);
        \node (f) at (1.5,+1) {$\Alt{F}$};
        \node (elf) at (4.5,+1) {$\El{\Alt{F}}$};
        \draw[->] (elf) to[bend left=22] (d);
        \node (rf) at (6,+1) {$\Res{F}$};
        
        \draw[->] (c) to[bend left=-63] (setop);
        \node (g) at (1.5,-1) {$\Alt{F}$};
        \node (elg) at (4.5,-1) {$\El{\Alt{F}}$};
        \draw[->] (elg) to[bend left=-22] (d);
        \node (rg) at (6,-1) {$\Res{F}$};
    
        \draw[nattrans] (f) to (g);
        \node at (1.5,0) {$\ID{F}$};
        \draw[double distance=2pt] (elf) to (elg);
        \node at (4.5,0) {$(\Id{\El{\Alt{F}}})^\po$};
        \draw[nattrans] (rf) to (rg);
        \node at (6,0) {$\ID{\Res{F}}$};
    \end{tikzpicture}
    \end{center}
\end{definition}

\begin{definition}\label{def:vertical-composition}
    Given two categories $\C$ and $\D$, three open functors $F, G, H : \C \oto \D$, and two open natural transformations $\theta : F \oTo G : \C \oto \D$ and $\phi : G \oTo H : \C \oto \D$, their \emph{open vertical composition $\phi \bullet \theta : F \oTo H : \C \oTo \D$} is the open natural transformation given by $\Alt{(\phi\bullet\theta)} = \Alt{\phi}\bullet\Alt{\theta}$ and $\Res{(\phi\bullet\theta)} = \Res{\phi} \bullet (\Res{\theta} \circ \ID{\El{(\Alt{\phi})^\po}})$.
\end{definition}

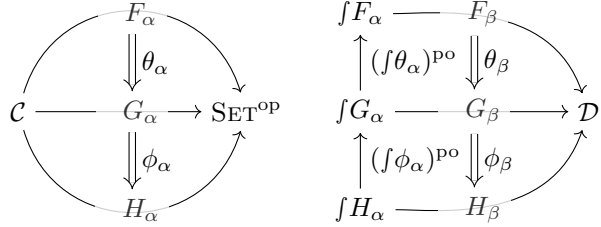
\begin{center}
\begin{tikzpicture}[on/.style={fill=white,fill opacity=0.8,text opacity=1}]
    \node (c)     at (0,0)   {$\C$};
    \node (setop) at (3,0)   {$\Set^\op$};
    \node (d)     at (7.5,0) {$\D$};

    \draw[->] (c) .. controls (0.5,+1.7) and (2.5,+1.7) .. (setop);
    \node[on] (f) at (1.5,+1.3) {$\Alt{F}$};
    \node (elf) at (4.5,+1.3) {$\El{\Alt{F}}$};
    \draw[->] (elf) .. controls (5.5,+1.3) and (7,1.5) .. (d);
    \node[on] (rf) at (6,+1.3) {$\Res{F}$};
    
    \draw[->](c) to (setop);
    \node[on] (g) at (1.5,0) {$\Alt{G}$};
    \node (elg) at (4.5,0) {$\El{\Alt{G}}$};
    \draw[->] (elg) to (d);
    \node[on] (rg) at (6,0) {$\Res{G}$};
    
    \draw[->] (c) .. controls (0.5,-1.7) and (2.5,-1.7) .. (setop);
    \node[on] (h) at (1.5,-1.3) {$\Alt{H}$};
    \node (elh) at (4.5,-1.3) {$\El{\Alt{H}}$};
    \draw[->] (elh) .. controls (5.5,-1.3) and (7,-1.5) .. (d);
    \node[on] (rh) at (6,-1.3) {$\Res{H}$};

    \draw[nattrans] (f) to node[right] {$\Alt{\theta}$} (g);
    \draw[style={<-}] (elf) to node[right]{$(\El{\Alt{\theta}})^\po$} (elg);
    \draw[nattrans] (rf) to node[right] {$\Res{\theta}$} (rg);
    
    \draw[nattrans] (g) to node[right] {$\Alt{\phi}$} (h);
    \draw[style={<-}] (elg) to node[right]{$(\El{\Alt{\phi}})^\po$} (elh);
    \draw[nattrans] (rg) to node[right] {$\Res{\phi}$} (rh);
\end{tikzpicture}
\end{center}

\begin{proposition}
    All identity open natural transformations and all open vertical compositions are indeed open natural transformations.
\end{proposition}
\begin{proof}
    Although obvious, one should still remember to check this.
\end{proof}

\begin{proposition}
    Given two categories $\C$ and $\D$, the collection $\OFun{\C}{\D}$ of all open functors forms a category where the collection of arrows between two open functors $F, G : \C \oto \D$ is given by the collection $\ONat{F}{G}$ of all open natural transformation, the composition by the open vertical composition $\bullet$ and the identity arrows by the identity open natural transformations.
\end{proposition}

\begin{proof}
    We need to prove that the open vertical composition is associative and has the identity open natural transformations as neutral elements.
    For the associativity, let us consider four open functors $F, G, H, I : \C \oto \D$, three open natural transformations $\theta : F \oTo G$, $\phi : G \oTo H$, and $\psi : H \oTo I$, and prove that $(\psi \bullet (\phi \bullet \theta)) = ((\psi \bullet \phi) \bullet \theta)$.
    \begin{align*}
        \Alt{(\psi \bullet (\phi \bullet \theta))} 
        & = \Alt{\psi} \bullet (\Alt{\phi} \bullet \Alt{\theta}) \text{ (by Def.~\ref{def:vertical-composition} two times)}\\
        & = (\Alt{\psi} \bullet \Alt{\phi}) \bullet \Alt{\theta}\\
        & = \Alt{((\psi \bullet \phi) \bullet \theta)} \text{ (by Def.~\ref{def:vertical-composition} two times)}
    \end{align*}
    \begin{align*}
         \Res{(\psi \bullet (\phi \bullet \theta))} 
        & = \Res{\psi} \bullet ((\Res{\phi} \bullet (\Res{\theta} \circ \ID{(\El{\Alt{\phi}})^\po})) \circ \ID{(\El{\Alt{\psi}})^\po}) \text{ (by Def.~\ref{def:vertical-composition} two times)}\\
        & = \Res{\psi} \bullet ((\Res{\phi} \circ \ID{(\El{\Alt{\psi}})^\po}) \bullet ((\Res{\theta} \circ \ID{(\El{\Alt{\phi}})^\po}) \circ \ID{(\El{\Alt{\psi}})^\po}))\\
        & = (\Res{\psi} \bullet (\Res{\phi} \circ \ID{(\El{\Alt{\psi}})^\po})) \bullet (\Res{\theta} \circ (\ID{(\El{\Alt{\phi}})^\po} \circ \ID{(\El{\Alt{\psi}})^\po}))\\
        & = (\Res{\psi} \bullet (\Res{\phi} \circ \ID{(\El{\Alt{\psi}})^\po})) \bullet (\Res{\theta} \circ \ID{(\El{\Alt{\phi}})^\po \circ (\El{\Alt{\psi}})^\po})\\
        & = (\Res{\psi} \bullet (\Res{\phi} \circ \ID{(\El{\Alt{\psi}})^\po})) \bullet (\Res{\theta} \circ \ID{\El{(\Alt{\psi} \bullet \Alt{\phi})^\po}})\\
        & = \Res{((\psi \bullet \phi) \bullet \theta)} \text{ (by Def.~\ref{def:vertical-composition} two times)}
    \end{align*}
    For the neutrality, consider two open functors $F, G : \C \oto \D$ and an open natural transformation $\theta : F \oTo G$.
    Let us prove that $\theta \bullet \ID{F} = \theta$.
    \begin{center}
    \vspace{-1.5em}
    \begin{minipage}[t]{.45\textwidth}
    \begin{align*}
        & \Alt{(\theta \bullet \ID{F})} \\
        & = \Alt{\theta} \bullet \Alt{\ID{F}} \text{ (by Def.~\ref{def:vertical-composition})} \\
        & = \Alt{\theta} \bullet \ID{\Alt{F}} \text{ (by Def.~\ref{def:identity-open-nat-trans})} \\
        & = \Alt{\theta}
    \end{align*}
    \end{minipage}
    \begin{minipage}[t]{.45\textwidth}
    \begin{align*}
        & \Res{(\theta \bullet \ID{F})} \\
        & = \Res{\theta} \bullet (\Res{\ID{F}} \circ \ID{(\El{\Alt{\theta}})^\po}) \text{ (by Def.~\ref{def:vertical-composition})}\\
        & = \Res{\theta} \bullet (\ID{\Res{F}} \circ \ID{(\El{\Alt{\theta}})^\po}) \text{ (by Def.~\ref{def:identity-open-nat-trans})}\\
        & = \Res{\theta} \bullet \ID{\Res{F} \circ (\El{\Alt{\theta}})^\po}\\
        & = \Res{\theta}
    \end{align*}
    \end{minipage}
    \end{center}
    Finally, let us prove that $\ID{F} \bullet \theta = \theta$.
    \begin{center}
    \vspace{-1.5em}
    \begin{minipage}[t]{.45\textwidth}
    \begin{align*}
        & \Alt{(\ID{G} \bullet \theta)} \\
        & = \Alt{\ID{G}} \bullet \Alt{\theta} \text{ (by Def.~\ref{def:vertical-composition})} \\
        & = \ID{\Alt{G}} \bullet \Alt{\theta} \text{ (by Def.~\ref{def:identity-open-nat-trans})} \\
        & = \Alt{\theta}
    \end{align*}
    \end{minipage}
    \begin{minipage}[t]{.45\textwidth}
    \begin{align*}
        & \Res{(\ID{G} \bullet \theta)} \\
        & = \Res{\ID{G}} \bullet (\Res{\theta} \circ \ID{(\El{\Alt{\ID{G}}})^\po}) \text{ (by Def.~\ref{def:vertical-composition})}\\
        & = \ID{\Res{G}} \bullet (\Res{\theta} \circ \ID{(\El{\ID{\Alt{G}}})^\po}) \text{ (by Def.~\ref{def:identity-open-nat-trans})}\\
        & = \Res{\theta} \circ \ID{(\Id{\El{\Alt{G}}})^\po}\\
        & = \Res{\theta}
    \end{align*}
    \end{minipage}
    \end{center}
\end{proof}

\begin{definition}\label{def:open-horizontal-composition}
    Given three categories $\C$, $\D$ and $\E$, four open functors $F, F' : \C \oto \D$ and $G, G' : \D \oto \E$ and two open natural transformations $\theta : F \oTo F' : \C \oto \D$ and $\phi : G \oTo G' : \D \oto \E$, their \emph{open horizontal composition $\phi \circ \theta : G \circ F \oTo G' \circ F' : \C \oto \E$} is the open natural transformation given by $\Alt{(\phi\circ\theta)}[c]^\po(\langle x', y' \rangle) = \langle\Alt{\theta}[c]^\po(x'),(\Alt{\phi} \circ \Res{\theta})[c,x']^\po(y')\rangle$ and $\Res{(\phi\circ\theta)}[c,\langle x',y' \rangle] = \Res{\phi}[\Res{\theta}[c,x'],y']$.
\end{definition}

\begin{center}
\begin{tikzpicture}[xscale=0.75,every node/.style={fill=white,fill opacity=0.8,text opacity=1}]
    \node (c)     at (0,0)   {$\C$};
    \node (setop) at (3,0)   {$\Set^\op$};
    \node (d)     at (7.5,0) {$\D$};
    \node (setop2) at (10.5,0)   {$\Set^\op$};
    \node (e)     at (15,0) {$\E$};

    \draw[->] (c) to[bend left=63] (setop);
    \node (f) at (1.5,+1) {$\Alt{F}$};
    \node (elf) at (4.5,+1) {$\El{\Alt{F}}$};
    \draw[->] (elf) to[bend left=22] (d);
    \node (rf) at (6,+1) {$\Res{F}$};
    
    \draw[->] (c) to[bend left=-63] (setop);
    \node (f2) at (1.5,-1) {$\Alt{F'}$};
    \node (elf2) at (4.5,-1) {$\El{\Alt{F'}}$};
    \draw[->] (elf2) to[bend left=-22] (d);
    \node (rf2) at (6,-1) {$\Res{F'}$};

    \draw[nattrans] (f) to (f2);
    \node at (1.5,0) {$\Alt{\theta}$};
    \draw[<-] (elf) to (elf2);
    \node at (4.5,0) {$(\El{\Alt{\theta}})^\po$};
    \draw[nattrans] (rf) to (rf2);
    \node at (6,0) {$\Res{\theta}$};

    \draw[->] (d) to[bend left=63] (setop2);
    \node (g) at (9,+1) {$\Alt{G}$};
    \node (elg) at (12,+1) {$\El{\Alt{G}}$};
    \draw[->] (elg) to[bend left=22] (e);
    \node (rg) at (13.5,+1) {$\Res{G}$};
    
    \draw[->] (d) to[bend left=-63] (setop2);
    \node (g2) at (9,-1) {$\Alt{G'}$};
    \node (elg2) at (12,-1) {$\El{\Alt{G'}}$};
    \draw[->] (elg2) to[bend left=-22] (e);
    \node (rg2) at (13.5,-1) {$\Res{G'}$};

    \draw[nattrans] (g) to (g2);
    \node at (9,0) {$\Alt{\phi}$};
    \draw[<-] (elg) to (elg2);
    \node at (12,0) {$(\El{\Alt{\phi}})^\po$};
    \draw[nattrans] (rg) to (rg2);
    \node at (13.5,0) {$\Res{\phi}$};
\end{tikzpicture}
\end{center}

\begin{remark}\label{rem:vertical-composition}
    The reader is invited to compare these formulas with those of $\Alt{(G \circ F)}(f)^\po(\langle x',y' \rangle)$ and $\Res{(G \circ F)}(f,\langle x',y' \rangle)$ in Def~\ref{def:open-composition}, as a small justification of the previous definition.
\end{remark}

\begin{proposition}
    All open horizontal compositions are indeed open natural transformations.
\end{proposition}
\begin{proof}
    Although obvious, one should still remember to check this.
\end{proof}

\begin{proposition}
    Given any three categories $\C$, $\D$ and $\E$, the composition of open functors together with the horizontal composition of open natural transformations form a functor from $\OFun{\D}{\E}\times\OFun{\C}{\D}$ to $\OFun{\C}{\E}$.
\end{proposition}
\begin{proof}
    We need to prove that identities and compositions are preserved.
    For the identities, let us consider two functors $F : \C \oto \D$ and $G : \D \oto \E$.
    The identity open natural transformation of $(G, F)$ in $\OFun{\D}{\E}\times\OFun{\C}{\D}$ is $(\ID{G}, \ID{F})$.
    We want to show that $\ID{G} \circ \ID{F} = \ID{G \circ F}$.
    \begin{align*}
        & \Alt{(\ID{G} \circ \ID{F})}[c]^\po(\langle x,y\rangle) \\
        & = \langle\Alt{\ID{F}}[c]^\po(x), (\Alt{\ID{G}} \circ \Res{\ID{F}})[c,x]^\po(y)\rangle \text{ (by Def.~\ref{def:open-horizontal-composition})}\\
        & = \langle\ID{\Alt{F}}[c]^\po(x), (\ID{\Alt{G}} \circ \ID{\Res{F}})[c,x]^\po(y)\rangle \text{ (by Def.~\ref{def:identity-open-nat-trans} three times)}\\
        & = \langle\id{\Alt{F}(c)}^\po(x), \id{(\Alt{G} \circ \Res{F})(c,x)}^\po(y)\rangle \\
        & = \langle x,y\rangle \\
        & = \id{\Alt{(G \circ F)}(c)}^\po(\langle x,y\rangle)\\
        & = \ID{\Alt{(G \circ F)}}[c]^\po(\langle x,y\rangle)\\
        & = \Alt{\ID{G \circ F}}[c]^\po(\langle x,y\rangle) \text{ (by Def.~\ref{def:identity-open-nat-trans})}
    \end{align*}
    \begin{align*}
        & \Res{(\ID{G} \circ \ID{F})}[c,\langle x,y\rangle] \\
        & = \Res{\ID{G}}[\Res{\ID{F}}[c,x],y] \text{ (by Def.~\ref{def:open-horizontal-composition})}\\
        & = \ID{\Res{G}}[\ID{\Res{F}}[c,x],y] \text{ (by Def.~\ref{def:identity-open-nat-trans} two times)}\\
        & = \ID{\Res{G}}[\id{\Res{F}(c,x)},y]\\
        & = \ID{\Res{G}}[\Res{F}(c,x),y]\\
        & = \id{\Res{G}(\Res{F}(c,x),y)}\\
        & = \id{\Res{(G \circ F)}(c,\langle x,y\rangle)} \text{ (by Def.~\ref{def:open-composition})}\\
        & = \Res{\ID{G \circ F}}[c,\langle x,y\rangle] \text{ (by Def.~\ref{def:identity-open-nat-trans})}
    \end{align*}
    Consider six functors $F,F',F'' : \C \oto \D$ and $G,G',G'' : \D \oto \E$ and four open natural transformations $\theta : F \oTo F'$, $\theta' : F' \oTo F''$, $\phi : G \oTo G'$ and $\phi' : G' \oTo G''$, as in the following diagram.
    \begin{center}
    \begin{tikzpicture}[xscale=0.75,every node/.style={fill=white,fill opacity=0.8,text opacity=1}]
        \node (c)     at (0,0)   {$\C$};
        \node (setop) at (3,0)   {$\Set^\op$};
        \node (d)     at (7.5,0) {$\D$};
        \node (setop2) at (10.5,0)   {$\Set^\op$};
        \node (e)     at (15,0) {$\E$};
    
        \draw[->] (c) .. controls (0.5,+2.5) and (2.5,+2.5) .. (setop);
        \node (f) at (1.5,+2) {$\Alt{F}$};
        \node (elf) at (4.5,+2) {$\El{\Alt{F}}$};
        \draw[->] (elf) .. controls (5.5,+2.2) and (7,+2) .. (d);
        \node (rf) at (6,+2) {$\Res{F}$};
        
        \draw[->] (c) to[bend left=0] (setop);
        \node (f2) at (1.5,0) {$\Alt{F'}$};
        \node (elf2) at (4.5,0) {$\El{\Alt{F'}}$};
        \draw[->] (elf2) to[bend left=0] (d);
        \node (rf2) at (6,0) {$\Res{F'}$};
    
        \draw[->] (c) .. controls (0.5,-2.5) and (2.5,-2.5) .. (setop);
        \node (f3) at (1.5,-2) {$\Alt{F''}$};
        \node (elf3) at (4.5,-2) {$\El{\Alt{F''}}$};
        \draw[->] (elf3) .. controls (5.5,-2.2) and (7,-2) .. (d);
        \node (rf3) at (6,-2) {$\Res{F''}$};
    
        \draw[nattrans] (f) to (f2);
        \node at (1.5,+1) {$\Alt{\theta}$};
        \draw[<-] (elf) to (elf2);
        \node at (4.5,+1) {$(\El{\Alt{\theta}})^\po$};
        \draw[nattrans] (rf) to (rf2);
        \node at (6,+1) {$\Res{\theta}$};
    
        \draw[nattrans] (f2) to (f3);
        \node at (1.5,-1) {$\Alt{\theta'}$};
        \draw[<-] (elf2) to (elf3);
        \node at (4.5,-1) {$(\El{\Alt{\theta'}})^\po$};
        \draw[nattrans] (rf2) to (rf3);
        \node at (6,-1) {$\Res{\theta'}$};
    
        \draw[->] (d) .. controls (8,+2.5) and (10,+2.5) .. (setop2);
        \node (g) at (9,+2) {$\Alt{G}$};
        \node (elg) at (12,+2) {$\El{\Alt{G}}$};
        \draw[->] (elg) .. controls (13,+2.2) and (14.5,+2) .. (e);
        \node (rg) at (13.5,+2) {$\Res{G}$};
        
        \draw[->] (d) to (setop2);
        \node (g2) at (9,0) {$\Alt{G'}$};
        \node (elg2) at (12,0) {$\El{\Alt{G'}}$};
        \draw[->] (elg2) to (e);
        \node (rg2) at (13.5,0) {$\Res{G'}$};
    
        \draw[->] (d) .. controls (8,-2.5) and (10,-2.5) .. (setop2);
        \node (g3) at (9,-2) {$\Alt{G''}$};
        \node (elg3) at (12,-2) {$\El{\Alt{G''}}$};
        \draw[->] (elg3) .. controls (13,-2.2) and (14.5,-2) .. (e);
        \node (rg3) at (13.5,-2) {$\Res{G''}$};

        \draw[nattrans] (g) to (g2);
        \node at (9,1) {$\Alt{\phi}$};
        \draw[<-] (elg) to (elg2);
        \node at (12,1) {$(\El{\Alt{\phi}})^\po$};
        \draw[nattrans] (rg) to (rg2);
        \node at (13.5,1) {$\Res{\phi}$};

        \draw[nattrans] (g2) to (g3);
        \node at (9,-1) {$\Alt{\phi'}$};
        \draw[<-] (elg2) to (elg3);
        \node at (12,-1) {$(\El{\Alt{\phi'}})^\po$};
        \draw[nattrans] (rg2) to (rg3);
        \node at (13.5,-1) {$\Res{\phi'}$};
    \end{tikzpicture}
    \end{center}
    We need to show that $(\phi' \bullet \phi) \circ (\theta' \bullet \theta) = (\phi' \circ \theta') \bullet (\phi \circ \theta)$.
    The equations being lengthy, we start by showing that both terms have $\Alt{(\--)}$ part which are pairs.  We develop the $\Alt{(\--)}$ part of the first.
    \begin{align*}
        & \Alt{((\phi' \bullet \phi) \circ (\theta' \bullet \theta))}[c]^\po(\langle x'',y''\rangle) \\
        & = \langle\Alt{(\theta' \bullet \theta)}[c]^\po(x''), (\Alt{(\phi' \bullet \phi)} \circ \Res{(\theta' \bullet \theta)})[c,x'']^\po(y'')\rangle \text{ (by Def.~\ref{def:open-horizontal-composition})}
    \end{align*}
    Now let us proceed similarly for the $\Alt{(\--)}$ part of second term.
    \begin{align*}
        & \Alt{((\phi' \circ \theta') \bullet (\phi \circ \theta))}[c]^\po(\langle x'',y'' \rangle) \\        
        & = (\Alt{(\phi' \circ \theta')} \bullet \Alt{(\phi \circ \theta)})[c]^\po(\langle x'',y'' \rangle) \text{ (by Def.~\ref{def:vertical-composition})} \\
        & = \Alt{(\phi \circ \theta)}[c]^\po(\Alt{(\phi' \circ \theta')}[c]^\po(\langle x'',y''\rangle)) \\
        & = \Alt{(\phi \circ \theta)}[c]^\po(\langle\Alt{\theta'}[c]^\po(x''),(\Alt{\phi'} \circ \Res{\theta'})[c,x'']^\po(y'')\rangle) \text{ (by Def.~\ref{def:open-horizontal-composition} here and below)}\\
        & = \langle\Alt{\theta}[c]^\po(\Alt{\theta'}[c]^\po(x'')),(\Alt{\phi} \circ \Res{\theta})[c,\Alt{\theta'}[c]^\po(x'')]^\po((\Alt{\phi'} \circ \Res{\theta'})[c,x'']^\po(y''))\rangle
    \end{align*}
    Now let us show that both pairs have the same first and second components.
    \begin{align*}
        & \Alt{(\theta' \bullet \theta)}[c]^\po(x'') \\
        & = (\Alt{\theta'} \bullet \Alt{\theta})[c]^\po(x'') \text{ (by Def.~\ref{def:vertical-composition})}\\
        & = \Alt{\theta}[c]^\po(\Alt{\theta'}[c]^\po(x''))
    \end{align*}
    \begin{align*}
        & (\Alt{(\phi' \bullet \phi)} \circ \Res{(\theta' \bullet \theta)})[c,x'']^\po(y'') \\        
        & = ((\Alt{\phi'} \bullet \Alt{\phi}) \circ (\Res{\theta'} \bullet (\Res{\theta} \circ \ID{(\El{\Alt{\theta'}})^\po})))[c,x'']^\po(y'') \text{ (by Def.~\ref{def:open-horizontal-composition})}\\
        & = ((\Alt{\phi'} \circ \Res{\theta'}) \bullet (\Alt{\phi} \circ (\Res{\theta} \circ \ID{(\El{\Alt{\theta'}})^\po})))[c,x'']^\po(y'') \\
        & = ((\Alt{\phi} \circ (\Res{\theta} \circ \ID{(\El{\Alt{\theta'}})^\po}))[c,x'']^\po \circ (\Alt{\phi'} \circ \Res{\theta'})[c,x'']^\po)(y'') \\
        & = (\Alt{\phi} \circ \Res{\theta})[c,\Alt{\theta'}[c]^\po(x'')]^\po((\Alt{\phi'} \circ \Res{\theta'})[c,x'']^\po(y''))
    \end{align*}
    Now, the $\Res{(\--)}$ part of both terms are compositions.
    \begin{align*}
        & \Res{((\phi' \bullet \phi) \circ (\theta' \bullet \theta))}[c,\langle x'', y''\rangle] \\
        & = \Res{(\phi' \bullet \phi)}[\Res{(\theta' \bullet \theta)}[c,x''],y''] \text{ (by Def.~\ref{def:open-horizontal-composition})}\\
        & = (\Res{\phi'} \bullet (\Res{\phi} \circ \ID{(\El{\Alt{\phi'}})^\po}))[(\Res{\theta'} \bullet (\Res{\theta} \circ \ID{(\El{\Alt{\theta'}})^\po}))[c,x''],y''] \text{ (by Def.~\ref{def:vertical-composition} two times)}\\
        & = (\Res{\phi'} \bullet (\Res{\phi} \circ \ID{(\El{\Alt{\phi'}})^\po}))[\Res{\theta'}[c,x''] \cdot \Res{\theta}[c,\Alt{\theta'}[c]^\po(x'')],y''] \\
        & = (\Res{\phi'} \bullet (\Res{\phi} \circ \ID{(\El{\Alt{\phi'}})^\po}))[(\Res{\theta'}[c,x''],y'') \cdot (\Res{\theta}[c,\Alt{\theta'}[c]^\po(x'')],\Alt{G''}(\Res{\theta'}[c,x''])^\po(y''))] \\
        & = \Res{\phi'}[\Res{\theta'}[c,x''],y''] \cdot (\Res{\phi} \circ \ID{(\El{\Alt{\phi'}})^\po})[\Res{\theta}[c,\Alt{\theta'}[c]^\po(x'')],\Alt{G''}(\Res{\theta'}[c,x''])^\po(y'')]
    \end{align*}
    \begin{align*}
        & \Res{((\phi' \circ \theta') \bullet (\phi \circ \theta))}[c,\langle x'', y''\rangle] \\
        & = (\Res{(\phi' \circ \theta')} \bullet (\Res{(\phi \circ \theta)} \circ \ID{(\El{\Alt{(\phi' \circ \theta')}})^\po}))[c,\langle x'', y''\rangle] \text{ (by Def.~\ref{def:vertical-composition})}\\
        & = \Res{(\phi' \circ \theta')}[c,\langle x'', y''\rangle] \cdot (\Res{(\phi \circ \theta)} \circ \ID{(\El{\Alt{(\phi' \circ \theta')}})^\po})[c,\langle x'', y''\rangle]
    \end{align*}
    Let us show that both compositions have the same first and second members.
    \begin{align*}
        & \Res{\phi'}[\Res{\theta'}[c,x''],y''] \\
        & = \Res{(\phi' \circ \theta')}[c,\langle x'', y''\rangle] \text{ (by Def.~\ref{def:open-horizontal-composition})}
    \end{align*}
    \begin{align*}
        & (\Res{\phi} \circ \ID{(\El{\Alt{\phi'}})^\po})[\Res{\theta}[c,\Alt{\theta'}[c]^\po(x'')],\Alt{G''}(\Res{\theta'}[c,x''])^\po(y'')] \\
        & = \Res{\phi}[\Res{\theta}[c,\Alt{\theta'}[c]^\po(x'')],\Alt{\phi'}[\Res{F'}[c,\Alt{\theta'}[c]^\po(x'')]]^\po(\Alt{G''}(\Res{\theta'}[c,x''])^\po(y''))] \\
        & = \Res{\phi}[\Res{\theta}[c,\Alt{\theta'}[c]^\po(x'')],(\Alt{\phi'}[\Res{F'}[c,\Alt{\theta'}[c]^\po(x'')]]^\po \circ \Alt{G''}(\Res{\theta'}[c,x'']^\po))(y'')] \\
        & = \Res{(\phi \circ \theta)}[c,\langle \Alt{\theta'}[c]^\po(x''), \Alt{\phi'}[\Res{\theta'}[c,x'']]^\po(y'') \rangle]\\
        & = \Res{(\phi \circ \theta)}[c,\langle \Alt{\theta'}[c]^\po(x''), (\Alt{\phi'} \circ \Res{\theta'})[c,x'']^\po(y'') \rangle] \text{ (by Def.~\ref{def:open-horizontal-composition})}\\
        & = \Res{(\phi \circ \theta)}[c,\Alt{(\phi' \circ \theta')}[c]^\po(\langle x'', y''\rangle)]  \text{ (by Def.~\ref{def:open-horizontal-composition})}\\
        & = (\Res{(\phi \circ \theta)} \circ \ID{(\El{\Alt{(\phi' \circ \theta')}})^\po})[c,\langle x'', y''\rangle]
    \end{align*}
\end{proof}

\begin{definition}\label{def:left-unitor}\label{def:right-unitor}
    Given two categories $\C$ and $\D$, their \emph{left unitor} is the (classical) natural isomorphism $\nl_{\C\D} : (\Id{\D} \circ \--) \To (\--) : \OFun{\C}{\D} \to \OFun{\C}{\D}$ with component $\nl_{\C\D F} : \Id{\D} \circ F \oTo F$ defined as $\Alt{\nl_{\C\D F}}[c]^\po(x) = \langle x,\star\rangle$ and $\Res{\nl_{\C\D F}}[c, x] = \id{\Res{F}(c,x)}$.
    Similarly, their \emph{right unitor} is the (classical) natural isomorphism $\nr_{\C\D} : (\-- \circ \Id{\C}) \To (\--) : \OFun{\C}{\D} \to \OFun{\C}{\D}$ with component $\nr_{\C\D F} : F \circ \Id{\C} \oTo F$ defined as $\Alt{\nr_{\C\D F}}[c]^\po(x) = \langle\star,x \rangle$ and $\Res{\nr_{\C\D F}}[c,x] = \id{\Res{F}(c,x)}$.
\end{definition}

\begin{proposition}
    These are indeed a natural isomorphisms.
\end{proposition}
\begin{proof}
    These maps are trivially isomorphism.
    Now, consider some open natural transformation $\theta : F \oTo G : \C \oto \D$.
    We need to prove that the naturality square commutes, \ie $\nl_{\C\D G} \bullet (\ID{\Id{\D}} \circ \theta) = \theta \bullet \nl_{\C\D F}$.
    For each component $\Alt{(\--)}$ and $\Res{(\--)}$, we show that both terms reduce to the same expression.
    \begin{align*}
        \Alt{(\theta \bullet \nl_{\C\D F})}[c]^\po(y)
        & = (\Alt{\theta} \bullet \Alt{\nl_{\C\D F}})[c]^\po(y) \text{ (by Def.~\ref{def:vertical-composition})}\\
        & = (\Alt{\nl_{\C\D F}}[c]^\po \circ \Alt{\theta}[c]^\po)(y) \\
        & = \Alt{\nl_{\C\D F}}[c]^\po(\Alt{\theta}[c]^\po(y)) \\
        & = \langle \Alt{\theta}[c]^\po(y),\star \rangle \text{ (by Def.~\ref{def:left-unitor})}
    \end{align*}
    \begin{align*}
        & \Alt{(\nl_{\C\D G} \bullet (\ID{\Id{\D}} \circ \theta))}[c]^\po(y) \\
        & = (\Alt{\nl_{\C\D G}} \bullet \Alt{(\ID{\Id{\D}} \circ \theta)})[c]^\po(y) \text{ (by Def.~\ref{def:vertical-composition})} \\
        & = (\Alt{(\ID{\Id{\D}} \circ \theta)}[c]^\po \circ \Alt{\nl_{\C\D G}}[c]^\po)(y) \\
        & = \Alt{(\ID{\Id{\D}} \circ \theta)}[c]^\po(\langle y,\star \rangle)  \text{ (by Def.~\ref{def:left-unitor})}\\
        & = \langle \Alt{\theta}[c]^\po(y), (\Alt{\ID{\Id{\D}}} \circ \Res{\theta})[c,y]^\po(\star) \rangle \text{ (by Def.~\ref{def:open-horizontal-composition})} \\
        & = \langle \Alt{\theta}[c]^\po(y), \star \rangle
    \end{align*}
    \begin{align*}
        \Res{(\theta \bullet \nl_{\C\D F})}[c,y]
        & = (\Res{\theta} \bullet (\Res{\nl_{\C\D F}} \circ \ID{(\El{\Alt{\theta}})^\po}))[c,y] \text{ (by Def.~\ref{def:vertical-composition})} \\
        & = \Res{\theta}[c,y] \cdot \Res{\nl_{\C\D F}}[c,\Alt{\theta}[c]^\po(y)] \\
        & = \Res{\theta}[c,y] \cdot \id{\Res{F}(c,\Alt{\theta}[c]^\po(y))} \text{ (by Def.~\ref{def:left-unitor})} \\
        & = \Res{\theta}[c,y]
    \end{align*}
    \begin{align*}
        & \Res{(\nl_{\C\D G} \bullet (\ID{\Id{\D}} \circ \theta))}[c,y] \\
        & = (\Res{\nl_{\C\D G}} \bullet (\Res{(\ID{\Id{\D}} \circ \theta)} \circ \ID{(\El{\Alt{\nl_{\C\D G}}})^\po}))[c,y] \text{ (by Def.~\ref{def:vertical-composition})} \\
        & = \Res{\nl_{\C\D G}}[c,y] \cdot \Res{(\ID{\Id{\D}} \circ \theta)}[c,\Alt{\nl_{\C\D G}}[c]^\po(y)] \\
        & = \id{\Res{G}(c,y)} \cdot \Res{(\ID{\Id{\D}} \circ \theta)}[c,\langle y, \star\rangle]  \text{ (by Def.~\ref{def:left-unitor} two times)}\\
        & = \Res{(\ID{\Id{\D}} \circ \theta)}[c,\langle y, \star\rangle] \\
        & = \ID{\Res{\Id{\D}}}[\Res{\theta}[c,y],\star] \text{ (by Def.~\ref{def:open-horizontal-composition} and~\ref{def:identity-open-nat-trans})} \\
        & = \Res{\theta}[c,y]  \text{ (by Def.~\ref{def:open-identity-functor})}
    \end{align*}
    Similarly for the right unitor, we need to prove that the naturality square commutes, \ie $\nr_{\C\D G} \bullet (\theta \circ \ID{\Id{\D}}) = \theta \bullet \nr_{\C\D F}$.
    \begin{align*}
        \Alt{(\theta \bullet \nr_{\C\D F})}[c]^\po(y)
        & = (\Alt{\theta} \bullet \Alt{\nr_{\C\D F}})[c]^\po(y) \text{ (by Def.~\ref{def:vertical-composition})}\\
        & = (\Alt{\nr_{\C\D F}}[c]^\po \circ \Alt{\theta}[c]^\po)(y) \\
        & = \Alt{\nr_{\C\D F}}[c]^\po(\Alt{\theta}[c]^\po(y)) \\
        & = \langle \star, \Alt{\theta}[c]^\po(y) \rangle \text{ (by Def.~\ref{def:left-unitor})}
    \end{align*}
    \begin{align*}
        & \Alt{(\nr_{\C\D G} \bullet (\theta \circ \ID{\Id{\D}}))}[c]^\po(y) \\
        & = (\Alt{\nr_{\C\D G}} \bullet \Alt{(\theta \circ \ID{\Id{\D}})})[c]^\po(y) \text{ (by Def.~\ref{def:vertical-composition})} \\
        & = (\Alt{(\theta \circ \ID{\Id{\D}})}[c]^\po \circ \Alt{\nr_{\C\D G}}[c]^\po)(y) \\
        & = \Alt{(\theta \circ \ID{\Id{\D}})}[c]^\po(\langle \star,y \rangle)  \text{ (by Def.~\ref{def:left-unitor})}\\
        & = \langle \Alt{\ID{\Id{\D}}}[c]^\po(\star), (\Alt{\theta} \circ \Res{\ID{\Id{\D}}})[c,y]^\po(\star) \rangle \text{ (by Def.~\ref{def:open-horizontal-composition})}  \\
        & = \langle \star, \Alt{\theta}[c]^\po(y) \rangle \text{ (by Def.~\ref{def:open-identity-functor} and~\ref{def:identity-open-nat-trans})}
    \end{align*}
    \begin{align*}
        \Res{(\theta \bullet \nr_{\C\D F})}[c,y]
        & = (\Res{\theta} \bullet (\Res{\nr_{\C\D F}} \circ \ID{(\El{\Alt{\theta}})^\po}))[c,y] \text{ (by Def.~\ref{def:vertical-composition})} \\
        & = \Res{\theta}[c,y] \cdot \Res{\nr_{\C\D F}}[c,\Alt{\theta}[c]^\po(y)] \\
        & = \Res{\theta}[c,y] \cdot \id{\Res{F}(c,\Alt{\theta}[c]^\po(y))} \text{ (by Def.~\ref{def:left-unitor})} \\
        & = \Res{\theta}[c,y]
    \end{align*}
    \begin{align*}
        & \Res{(\nr_{\C\D G} \bullet (\theta \circ \ID{\Id{\D}}))}[c,y] \\
        & = (\Res{\nr_{\C\D G}} \bullet (\Res{(\theta \circ \ID{\Id{\D}})} \circ \ID{(\El{\Alt{\nr_{\C\D G}}})^\po}))[c,y] \text{ (by Def.~\ref{def:vertical-composition})} \\
        & = \Res{\nr_{\C\D G}}[c,y] \cdot \Res{(\theta \circ \ID{\Id{\D}})}[c,\Alt{\nr_{\C\D G}}[c]^\po(y)] \\
        & = \id{\Res{G}(c,y)} \cdot \Res{(\theta \circ \ID{\Id{\D}})}[c,\langle \star, y \rangle]  \text{ (by Def.~\ref{def:left-unitor} two times)}\\
        & = \Res{(\theta \circ \ID{\Id{\D}})}[c,\langle \star, y \rangle] \\
        & = \Res{\theta}[\ID{\Res{\Id{\D}}}[c,\star],y] \text{ (by Def.~\ref{def:open-horizontal-composition} and~\ref{def:identity-open-nat-trans})} \\
        & = \Res{\theta}[\Res{\Id{\D}}(c,\star),y] \\
        & = \Res{\theta}[c,y] \text{ (by Def.~\ref{def:open-identity-functor})}
    \end{align*}
\end{proof}

\begin{definition}\label{def:associator}
    Given four open categories $\B$, $\C$, $\D$ and $\E$, and three open functors $F : \B \oto \C$, $G : \C \oto \D$, and $H: \D \oto \E$, their \emph{associator} is the (classical) natural isomorphism $\na_{\B\C\D\E} : (\-- \circ (\-- \circ \--)) \to ((\-- \circ \--) \circ \--) : \OFun{\D}{\E} \times\OFun{\C}{\D} \times\OFun{\B}{\C} \to \OFun{\B}{\E}$ with component $\na_{\B\C\D\E}[F,G,H] : H \circ (G \circ F) \oto (H \circ G) \circ F : \B \oto \E$ defined as $\Alt{(\na_{\B\C\D\E}[F,G,H])}[b]^\po(\langle x,\langle y,z\rangle\rangle) = \langle\langle x,y\rangle,z\rangle$ and $\Res{(\na_{\B\C\D\E}[F,G,H])}[b,\langle x,\langle y,z\rangle\rangle] = \id{\Res{H}(\Res{G}(\Res{F}(b,x),y),z)}$ (or equivalently $\Res{(\na_{\B\C\D\E}[F,G,H])} = \ID{\Res{((H \circ G) \circ F)}}$).
\end{definition}
\begin{proposition}
    This is indeed a natural isomorphism.
\end{proposition}
\begin{proof}
    These maps are trivially isomorphism.
    Now, consider three open natural transformations $\theta : F \oTo F' : \B \oto \C$, $\phi : G \oTo G' : \C \oto \D$, and $\psi : H \oTo H' : \D \oto \E$.
    We need to prove that the naturality square commutes, \ie $\na' \bullet (\psi \circ (\phi \circ \theta)) = ((\psi \circ \phi) \circ \theta) \bullet \na$ where $\na$ stands for $\na_{\B\C\D\E}[F,G,H]$ and $\na'$ for $\na_{\B\C\D\E}[F',G',H']$.
    \begin{align*}
        & \Alt{(\na' \bullet (\psi \circ (\phi \circ \theta)))}[b]^\po(\langle x', \langle y', z' \rangle \rangle) \\
        & = (\Alt{(\psi \circ (\phi \circ \theta))}[b]^\po \circ \Alt{\na'}[b]^\po)(\langle x', \langle y', z' \rangle \rangle) \text{ (by Def.~\ref{def:vertical-composition})}\\
        & = \Alt{(\psi \circ (\phi \circ \theta))}[b]^\po(\Alt{\na'}[b]^\po(\langle x', \langle y', z' \rangle \rangle))\\
        & = \Alt{(\psi \circ (\phi \circ \theta))}[b]^\po(\langle \langle x',  y' \rangle, z'  \rangle) \text{ (by Def.~\ref{def:associator})}\\
        & = \langle \Alt{(\phi \circ \theta)}[b]^\po (\langle x',  y' \rangle), (\Alt{\psi} \circ \Res{(\phi \circ \theta)})[b,\langle x',  y' \rangle]^\po(z') \rangle \text{ (by Def.~\ref{def:open-horizontal-composition})}\\
        & = \langle \Alt{(\phi \circ \theta)}[b]^\po (\langle x',  y' \rangle), \Alt{\psi}[\Res{(\phi \circ \theta)}[b,\langle x',  y' \rangle]]^\po(z') \rangle \\
        & = \langle \langle \Alt{\theta}[b]^\po(x'), (\Alt{\phi} \circ \Res{\theta})[b,x']^\po(y') \rangle, \Alt{\psi}[\Res{\phi}[\Res{\theta}[b,x'],y']]^\po(z') \rangle \text{ (by Def.~\ref{def:open-horizontal-composition})} \\
        & = \langle \langle \Alt{\theta}[b]^\po(x'), \Alt{\phi}[\Res{\theta}[b,x']]^\po(y') \rangle, \Alt{\psi}[\Res{\phi}[\Res{\theta}[b,x'],y']]^\po(z') \rangle
    \end{align*}
    \begin{align*}
        & \Alt{(((\psi \circ \phi) \circ \theta) \bullet \na)}[b]^\po(\langle x', \langle y', z' \rangle \rangle) \\
        & = (\Alt{\na}[b]^\po \circ \Alt{((\psi \circ \phi) \circ \theta)}[b]^\po)(\langle x', \langle y', z' \rangle \rangle) \text{ (by Def.~\ref{def:vertical-composition})}\\
        & = \Alt{\na}[b]^\po(\Alt{((\psi \circ \phi) \circ \theta)}[b]^\po(\langle x', \langle y', z' \rangle \rangle)) \\
        & = \Alt{\na}[b]^\po(\langle \Alt{\theta}[b]^\po(x'), (\Alt{(\psi \circ \phi)} \circ \Res{\theta})[b,x']^\po(\langle y', z' \rangle) \rangle) \text{ (by Def.~\ref{def:open-horizontal-composition})}\\
        & = \Alt{\na}[b]^\po(\langle \Alt{\theta}[b]^\po(x'), \Alt{(\psi \circ \phi)}[\Res{\theta}[b,x']]^\po(\langle y', z' \rangle) \rangle) \text{ (and Def.~\ref{def:open-horizontal-composition} gives)}\\
        & = \Alt{\na}[b]^\po(\langle \Alt{\theta}[b]^\po(x'), \langle \Alt{\phi}[\Res{\theta}[b,x']]^\po(y'), (\Alt{\psi} \circ \Res{\phi})[\Res{\theta}[b,x'],y']^\po(z') \rangle \rangle) \\
        & = \Alt{\na}[b]^\po(\langle \Alt{\theta}[b]^\po(x'), \langle \Alt{\phi}[\Res{\theta}[b,x']]^\po(y'), \Alt{\psi}[\Res{\phi}[\Res{\theta}[b,x'],y']]^\po(z') \rangle \rangle) \\
        & = \langle \langle \Alt{\theta}[b]^\po(x'), \Alt{\phi}[\Res{\theta}[b,x']]^\po(y') \rangle, \Alt{\psi}[\Res{\phi}[\Res{\theta}[b,x'],y']]^\po(z') \rangle \text{ (by Def.~\ref{def:associator})}
    \end{align*}
    \begin{align*}
        & \Res{(\na' \bullet (\psi \circ (\phi \circ \theta)))}[b,\langle x', \langle y', z' \rangle \rangle]\\
        & = (\Res{\na'} \bullet (\Res{(\psi \circ (\phi \circ \theta))} \circ \ID{(\El{\Alt{\na'}})^\po}))[b,\langle x', \langle y', z' \rangle \rangle] \text{ (by Def.~\ref{def:vertical-composition})}\\
        & = \Res{\na'}[b,\langle x', \langle y', z' \rangle \rangle] \cdot (\Res{(\psi \circ (\phi \circ \theta))} \circ \ID{(\El{\Alt{\na'}})^\po})[b,\langle x', \langle y', z' \rangle \rangle] \\
        & = \id{\ldots} \cdot (\Res{(\psi \circ (\phi \circ \theta))} \circ \ID{(\El{\Alt{\na'}})^\po})[b,\langle x', \langle y', z' \rangle \rangle] \text{ (by Def.~\ref{def:associator})}\\
        & = (\Res{(\psi \circ (\phi \circ \theta))} \circ \ID{(\El{\Alt{\na'}})^\po})[b,\langle x', \langle y', z' \rangle \rangle] \\
        & = \Res{(\psi \circ (\phi \circ \theta))}[b,\Alt{\na'}[b]^\po(\langle x', \langle y', z' \rangle \rangle)] \\
        & = \Res{(\psi \circ (\phi \circ \theta))}[b,\langle \langle x', y' \rangle, z' \rangle] \text{ (by Def.~\ref{def:associator})}\\
        & = \Res{\psi}[\Res{(\phi \circ \theta)}[b,\langle x', y' \rangle], z'] \text{ (by Def.~\ref{def:open-horizontal-composition})}\\
        & = \Res{\psi}[\Res{\phi}[\Res{\theta}[b,x'], y'], z'] \text{ (by Def.~\ref{def:open-horizontal-composition})}
    \end{align*}
    \begin{align*}
        & \Res{(((\psi \circ \phi) \circ \theta) \bullet \na)}[b,\langle x', \langle y', z' \rangle \rangle] \\
        & = (\Res{((\psi \circ \phi) \circ \theta)} \bullet (\Res{\na} \circ \ID{(\El{\Alt{((\psi \circ \phi) \circ \theta)}})^\po}))[b,\langle x', \langle y', z' \rangle \rangle] \text{ (by Def.~\ref{def:vertical-composition})}\\
        & = \Res{((\psi \circ \phi) \circ \theta)}[b,\langle x', \langle y', z' \rangle \rangle] \cdot (\Res{\na} \circ \ID{(\El{\Alt{((\psi \circ \phi) \circ \theta)}})^\po})[b,\langle x', \langle y', z' \rangle \rangle] \\
        & = \Res{((\psi \circ \phi) \circ \theta)}[b,\langle x', \langle y', z' \rangle \rangle] \cdot \Res{\na}[b,\Alt{((\psi \circ \phi) \circ \theta)}[b]^\po(\langle x', \langle y', z' \rangle \rangle)] \\
        & = \Res{((\psi \circ \phi) \circ \theta)}[b,\langle x', \langle y', z' \rangle \rangle] \cdot \id{\ldots} \text{ (by Def.~\ref{def:associator})}\\
        & = \Res{((\psi \circ \phi) \circ \theta)}[b,\langle x', \langle y', z' \rangle \rangle] \\
        & = \Res{(\psi \circ \phi)}[\Res{\theta}[b,x'], \langle y', z' \rangle] \text{ (by Def.~\ref{def:open-horizontal-composition})}\\
        & = \Res{\psi}[\Res{\phi}[\Res{\theta}[b,x'],y'],z'] \text{ (by Def.~\ref{def:open-horizontal-composition})}
    \end{align*}
\end{proof}
    
\begin{proposition}
    The collection of all categories together with open functors and open natural transformations, in the precise sense defined up to now, form a bicategory.
\end{proposition}
\begin{proof}
    The only remaining tasks are to prove the coherence of the associator via the pentagon identity, and the coherence of the left and right unitors via the triangle identity.
    Let us begin by the pentagon identity, which amounts to the commutativity of the following diagram for any categories $\A$,$\B$,$\C$,$\D$, and $\E$ and any open functors $F : \A \oto \B$, $G : \B \oto \C$, $H : \C \oto \D$, $I : \D \oto \E$.
    \begin{center}
    \begin{tikzpicture}
    \node (a) at (0,3) {$I \circ (H \circ (G \circ F))$};
    \node (b) at (0,1.5) {$(I \circ H) \circ (G \circ F)$};
    \node (c) at (8,3) {$I \circ ((H \circ G) \circ F)$};
    \node (d) at (8,1.5) {$(I \circ (H \circ G)) \circ F$};
    \node (e) at (4,0) {$((I \circ H) \circ G) \circ F$};
    \draw[onattrans] (a) to node[right]{$\na[G\circ F,H,I]$} (b);
    \draw[onattrans] (b) to node[below left]{$\na[F,G,I \circ H]$} (e);
    \draw[onattrans] (a) to node[above]{$\ID{I}\circ \na[F,G,H]$} (c);
    \draw[onattrans] (c) to node[left]{$\na[F,H \circ G,I]$} (d);
    \draw[onattrans] (d) to node[below right]{$\na[I,H,G] \circ \ID{F}$} (e);
    \end{tikzpicture}
    \end{center}
    For the $\Alt{(\--)}$ components first and then for the $\Res{(\--)}$ components, let us show that both paths reduces to the same expression.
    For any $k \in A$ and $e \in \Alt{((I \circ H) \circ G) \circ F)}$ we have the following for the three step path.
    \begin{align*}
        & \Alt{((\na[G,H,I] \circ \ID{F}) \bullet \na[F,H\circ G,I] \bullet (\ID{I} \circ \na[F,G,H]))}[k]^\po(e) \\
        & = \Alt{(\na[G,H,I] \circ \ID{F})} \bullet \Alt{\na[F,H\circ G,I]} \bullet \Alt{(\ID{I} \circ \na[F,G,H])})[k]^\po(e) \text{ (by Def.~\ref{def:vertical-composition})}\\
        & = \Alt{(\ID{I} \circ \na[F,G,H])}[k]^\po \circ \Alt{\na[F,H\circ G,I]}[k]^\po \circ \Alt{(\na[G,H,I] \circ \ID{F})}[k]^\po(e)
    \end{align*}
    Now $e$ has the form $\langle w , \langle x , \langle y , z \rangle \rangle \rangle$ for some $w \in \Alt{F}$, $x \in \Alt{G}$, $y \in \Alt{H}$, $z \in \Alt{I}$.
    Let us show that the functions in the last line above simply does what we expect.
    \begin{align*}
        & \Alt{(\na[G,H,I] \circ \ID{F})}[k]^\po(\langle w , \langle x , \langle y , z \rangle \rangle \rangle) \\
        & = \langle \Alt{\ID{F}}[k]^\po(w), (\Alt{\na[G,H,I]} \circ \Res{\ID{F}})[k,w]^\po(\langle x , \langle y , z \rangle \rangle) \rangle) \text{ (by Def.~\ref{def:open-horizontal-composition})}\\
        & = \langle \Alt{\ID{F}}[k]^\po(w), \Alt{\na[G,H,I]}[\Res{\ID{F}}[k,w]]^\po(\langle x , \langle y , z \rangle \rangle) \rangle)\\
        & = \langle \Alt{\ID{F}}[k]^\po(w), \Alt{\na[G,H,I]}[\id{\Res{F}(k,w)}]^\po(\langle x , \langle y , z \rangle \rangle) \rangle) \text{ (by Def.~\ref{def:identity-open-nat-trans})}\\
        & = \langle \Alt{\ID{F}}[k]^\po(w), \Alt{\na[G,H,I]}[\Res{F}(k,w)]^\po(\langle x , \langle y , z \rangle \rangle) \rangle)\\
        & = \langle \Alt{\ID{F}}[k]^\po(w), \langle \langle x , y \rangle , z \rangle \rangle \text{ (by Def.~\ref{def:associator})}\\
        & = \langle w, \langle \langle x , y \rangle , z \rangle \rangle \text{ (by Def.~\ref{def:identity-open-nat-trans})}
    \end{align*}
    \begin{align*}
        & \Alt{\na[F,H\circ G,I]}[k]^\po(\langle w, \langle \langle x , y \rangle , z \rangle \rangle) \\
        & = \langle \langle w, \langle x , y \rangle \rangle , z \rangle \text{ (by Def.~\ref{def:associator})}
    \end{align*}
    \begin{align*}
        & \Alt{(\ID{I} \circ \na[F,G,H])}[k]^\po(\langle \langle w, \langle x , y \rangle \rangle , z \rangle) \\
        & = \langle \Alt{\na[F,G,H]}[k]^\po(\langle w, \langle x , y \rangle \rangle) , (\Alt{\ID{I}}\circ \Res{\na[F,G,H]})[k,\langle w, \langle x , y \rangle \rangle]^\po(z) \rangle \text{ (by Def.~\ref{def:open-horizontal-composition})} \\
        & = \langle \langle \langle w, x \rangle , y \rangle , (\Alt{\ID{I}}\circ \Res{\na[F,G,H]})[k,\langle w, \langle x , y \rangle \rangle]^\po(z) \rangle \text{ (by Def.~\ref{def:associator})} \\
        & = \langle \langle \langle w, x \rangle , y \rangle , \Alt{\ID{I}}[\Res{\na[F,G,H]}[k,\langle w, \langle x , y \rangle \rangle]]^\po(z) \rangle \\
        & = \langle \langle \langle w, x \rangle , y \rangle , \Alt{\ID{I}}[\id{\Res{H}(\Res{G}(\Res{F}(k,w),x),y)}]^\po(z) \rangle \text{ (by Def.~\ref{def:associator})} \\
        & = \langle \langle \langle w, x \rangle , y \rangle , \Alt{\ID{I}}[\Res{H}(\Res{G}(\Res{F}(k,w),x),y)]^\po(z) \rangle\\
        & = \langle \langle \langle w, x \rangle , y \rangle , z \rangle \text{ (by Def.~\ref{def:identity-open-nat-trans})} \\
    \end{align*}
    The $\Alt{(\--)}$ component of the two step path indeed leads to the same result.
    \begin{align*}
        & \Alt{(\na[F,G,I\circ H] \bullet \na[G\circ F,H,I])}[k]^\po(\langle w , \langle x , \langle y , z \rangle \rangle \rangle) \\
        & = (\Alt{\na[F,G,I\circ H]} \bullet \Alt{\na[G\circ F,H,I]})[k]^\po(\langle w , \langle x , \langle y , z \rangle \rangle \rangle) \text{ (by Def.~\ref{def:vertical-composition})}\\
        & = (\Alt{\na[G\circ F,H,I]}[k]^\po \circ \Alt{\na[F,G,I\circ H]}[k]^\po)(\langle w , \langle x , \langle y , z \rangle \rangle \rangle)\\
        & = \Alt{\na[G\circ F,H,I]}[k]^\po (\Alt{\na[F,G,I\circ H]}[k]^\po(\langle w , \langle x , \langle y , z \rangle \rangle \rangle))\\
        & = \Alt{\na[G\circ F,H,I]}[k]^\po (\langle \langle w , x \rangle , \langle y , z \rangle \rangle) \text{ (by Def.~\ref{def:associator})}\\
        & = \langle \langle \langle w , x \rangle , y \rangle , z \rangle \text{ (by Def.~\ref{def:associator})}
    \end{align*}
    Now let us compute the $\Res{(\--)}$ component of the three step path, step by step, using freely the previously established equations.
    \begin{align*}
        & \Res{((\na[G,H,I] \circ \ID{F}) \bullet \theta)}[k,\langle w , \langle x , \langle y , z \rangle \rangle \rangle] \text{ (and Def.~\ref{def:vertical-composition} gives)}\\
        & = (\Res{(\na[G,H,I] \circ \ID{F})} \bullet (\Res{\theta} \circ \ID{(\El{\Alt{(\na[G,H,I] \circ \ID{F})}})^\po}))[k,\langle w , \langle x , \langle y , z \rangle \rangle \rangle] \\
        & = \Res{(\na[G,H,I] \circ \ID{F})}[k,\langle w , \langle x , \langle y , z \rangle \rangle \rangle] \cdot \Res{\theta}[k,\langle w , \langle \langle x , y \rangle , z \rangle \rangle] \\
        & =  \Res{\na[G,H,I]}[\Res{\ID{F}}[k,w], \langle x , \langle y , z \rangle \rangle]] \cdot \Res{\theta}[k,\langle w , \langle \langle x , y \rangle , z \rangle \rangle] \text{ (by Def.~\ref{def:open-horizontal-composition})}\\
        & =  \Res{\na[G,H,I]}[\id{\Res{F}(k,w)}, \langle x , \langle y , z \rangle \rangle]] \cdot \Res{\theta}[k,\langle w , \langle \langle x , y \rangle , z \rangle \rangle] \text{ (by Def.~\ref{def:identity-open-nat-trans})}\\
        & =  \Res{\na[G,H,I]}[\Res{F}(k,w), \langle x , \langle y , z \rangle \rangle]] \cdot \Res{\theta}[k,\langle w , \langle \langle x , y \rangle , z \rangle \rangle] \\
        & =  \id{\Res{I}(\Res{H}(\Res{G}(\Res{F}(k,w),x),y),z)} \cdot \Res{\theta}[k,\langle w , \langle \langle x , y \rangle , z \rangle \rangle] \text{ (by Def.~\ref{def:associator})}\\
        & =  \Res{\theta}[k,\langle w , \langle \langle x , y \rangle , z \rangle \rangle]
    \end{align*}
    \begin{align*}
        & \Res{(\na[F,H\circ G,I] \bullet \theta')}[k,\langle w , \langle \langle x , y \rangle , z \rangle \rangle] \\
        & = (\Res{\na[F,H\circ G,I]} \bullet (\Res{\theta'} \circ \ID{(\El{\Alt{\na[F,H\circ G,I]}})^\po}))[k,\langle w , \langle \langle x , y \rangle , z \rangle \rangle] \text{ (by Def.~\ref{def:vertical-composition})}\\
        & = \Res{\na[F,H\circ G,I]}[k,\langle w , \langle \langle x , y \rangle , z \rangle \rangle] \cdot \Res{\theta'}[k,\langle \langle w , \langle x , y \rangle \rangle , z \rangle] \\
        & = \id{\Res{I}(\Res{(H \circ G)}(\Res{F}(k,w), \langle x , y \rangle, z)} \cdot \Res{\theta'}[k,\langle \langle w , \langle x , y \rangle \rangle , z \rangle] \text{ (by Def.~\ref{def:associator})}\\
        & = \Res{\theta'}[k,\langle \langle w , \langle x , y \rangle \rangle , z \rangle]
    \end{align*}
    \begin{align*}
        & \Res{(\ID{I} \circ \na[F,G,H])}[k,\langle \langle w , \langle x , y \rangle \rangle , z \rangle] \\
        & = \Res{\ID{I}}[\Res{\na[F,G,H]}[k,\langle w , \langle x , y \rangle \rangle], z] \text{ (by Def.~\ref{def:open-horizontal-composition})}\\
        & = \Res{\ID{I}}[\id{\Res{H}(\Res{G}(\Res{F}(k,w),x),y)], z}] \text{ (by Def.~\ref{def:associator})}\\
        & = \Res{\ID{I}}[\Res{H}(\Res{G}(\Res{F}(k,w),x),y),z] \\
        & = \id{\Res{I}(\Res{H}(\Res{G}(\Res{F}(k,w),x),y),z)} \text{ (by Def.~\ref{def:identity-open-nat-trans})}
    \end{align*}
    Combining these three equations by letting $\theta$ of the first equation be the second equation and $\theta'$ of the second equation be the third equation, we obtain the following one.
    \begin{align*}
        & \Res{((\na[G,H,I] \circ \ID{F}) \bullet (\na[F,H\circ G,I] \bullet (\ID{I} \circ \na[F,G,H])))}[k,\langle w , \langle x , \langle y , z \rangle \rangle \rangle] \\
        & =  \Res{(\na[F,H\circ G,I] \bullet (\ID{I} \circ \na[F,G,H]))}[k,\langle w , \langle \langle x , y \rangle , z \rangle \rangle] \\
        & = \Res{(\ID{I} \circ \na[F,G,H])}[k,\langle \langle w , \langle x , y \rangle \rangle , z \rangle] \\
        & = \id{\Res{I}(\Res{H}(\Res{G}(\Res{F}(k,w),x),y),z)}
    \end{align*}
    Now let us use the same strategy for the $\Res{(\--)}$ components of the two step path.
    \begin{align*}
        & \Res{(\na[F,G,I\circ H] \bullet \theta)}[k,\langle w , \langle x , \langle y , z \rangle \rangle \rangle] \\
        & = (\Res{\na[F,G,I\circ H]} \bullet (\Res{\theta} \circ \ID{(\El{\Alt{\na[F,G,I\circ H]}})^\po}))[k,\langle w , \langle x , \langle y , z \rangle \rangle \rangle] \text{ (by Def.~\ref{def:vertical-composition})}\\
        & = \Res{\na[F,G,I\circ H]}[k,\langle w , \langle x , \langle y , z \rangle \rangle \rangle] \cdot \Res{\theta}[k,\langle \langle w , x \rangle , \langle y , z \rangle \rangle] \\
        & = \id{\Res{(I \circ H)}(\Res{G}(\Res{F}(k,w), x) , \langle y, z \rangle)} \cdot \Res{\theta}[k,\langle \langle w , x \rangle , \langle y , z \rangle \rangle] \text{ (by Def.~\ref{def:associator})}\\
        & = \Res{\theta}[k,\langle \langle w , x \rangle , \langle y , z \rangle \rangle]
    \end{align*}
    \begin{align*}
        & \Res{\na[G \circ F,H,I]}[k,\langle \langle w , x \rangle , \langle y , z \rangle \rangle] \\
        & = \id{\Res{I}(\Res{H}(\Res{G}(\Res{F}(k,w),x),y),z)} \text{ (by Def.~\ref{def:associator})}
    \end{align*}
    Combining these two steps we obtain the desired equation.
    We are left with a similar task for the triangle identity, \ie the condition that the following diagram commutes.
    \begin{center}
    \begin{tikzpicture}
    \node (a) at (0,1.5) {$G \circ (\Id{\C} \circ F)$};
    \node (b) at (8,1.5) {$(G \circ \Id{\C}) \circ F$};
    \node (c) at (4,0) {$G \circ F$};
    \draw[onattrans] (a) to node[above]{$\na[F,\Id{\C},G]$} (b);
    \draw[onattrans] (a) to node[below left]{$\ID{G} \circ \nl[F]$} (c);
    \draw[onattrans] (b) to node[below right]{$\nr[G] \circ \ID{F}$} (c);
    \end{tikzpicture}
    \end{center}
    Let us start with the $\Alt{(\--)}$ components of the two-steps path.
    \begin{align*}
        & \Alt{(\nr[G] \circ \ID{F})}[k]^\po(\langle x, y \rangle) \\
        & = \langle \Alt{\ID{F}}[k]^\po(x), (\Alt{\nr[G]} \circ \Res{\ID{F}})[k,x]^\po(y) \rangle \text{ (by Def.~\ref{def:open-horizontal-composition})}\\
        & = \langle x, (\Alt{\nr[G]} \circ \Res{\ID{F}})[k,x]^\po(y) \rangle \text{ (by Def.~\ref{def:identity-open-nat-trans})}\\
        & = \langle x, \Alt{\nr[G]}[\Res{\ID{F}}[k,x]]^\po(y) \rangle \\
        & = \langle x, \Alt{\nr[G]}[\Id{\Res{F(k,x)}}]^\po(y) \rangle \text{ (by Def.~\ref{def:identity-open-nat-trans})}\\
        & = \langle x, \Alt{\nr[G]}[\Res{F(k,x)}]^\po(y) \rangle \\
        & = \langle x, \langle \star, y \rangle \rangle\text{ (by Def.~\ref{def:right-unitor})}
    \end{align*}
    \begin{align*}
        & \Alt{\na[F,\Id{\C},G]}[k]^\po(\langle x, \langle \star, y \rangle \rangle) \\
        & = \langle \langle x,  \star \rangle, y \rangle\text{ (by Def.~\ref{def:associator})}
    \end{align*}
    Now the $\Alt{(\--)}$ components of the one step path.
    \begin{align*}
        & \Alt{(\ID{G} \circ \nl[F])}[k]^\po(\langle x, y \rangle) \\
        & = \langle \Alt{\nl[F]}[k]^\po(x), (\Alt{\ID{G}}\circ\Res{\nl[F]})[k,x]^\po(y) \rangle\text{ (by Def.~\ref{def:open-horizontal-composition})} \\
        & = \langle \Alt{\nl[F]}[k]^\po(x), \Alt{\ID{G}}[\Res{\nl[F]}[k,x]]^\po(y) \rangle \\
        & = \langle \Alt{\nl[F]}[k]^\po(x), \Alt{\ID{G}}[\id{F(k,x)}]^\po(y) \rangle\text{ (by Def.~\ref{def:left-unitor})} \\
        & = \langle \Alt{\nl[F]}[k]^\po(x), \Alt{\ID{G}}[F(k,x)]^\po(y) \rangle \\
        & = \langle \Alt{\nl[F]}[k]^\po(x), y \rangle\text{ (by Def.~\ref{def:identity-open-nat-trans})}\\
        & = \langle \langle x, \star \rangle, y \rangle \text{ (by Def.~\ref{def:left-unitor})}\\
    \end{align*}
    Both paths agree on the $\Alt{(\--)}$ component.
    Now let us calculate the $\Res{(\--)}$ component of the two-steps path.
    \begin{align*}
        & \Res{((\nr[G] \circ \ID{F}) \bullet \theta)}[k,\langle x, y \rangle] \\
        & = (\Res{(\nr[G] \circ \ID{F})} \bullet (\Res{\theta} \circ \ID{(\El{\Alt{(\nr[G] \circ \ID{F})})^\po}}))[k,\langle x, y \rangle] \text{ (by Def.~\ref{def:vertical-composition})}\\
        & = \Res{(\nr[G] \circ \ID{F})}[k,\langle x, y \rangle] \cdot \Res{\theta}[k,\langle x, \langle \star, y \rangle \rangle] \\
        & = \Res{\nr[G]}[\Res{\ID{F}}[k,x], y]] \cdot \Res{\theta}[k,\langle x, \langle \star, y \rangle \rangle] \text{ (by Def.~\ref{def:open-horizontal-composition})}\\
        & = \Res{\nr[G]}[\id{\Res{F}(k,x)}, y]] \cdot \Res{\theta}[k,\langle x, \langle \star, y \rangle \rangle] \text{ (by Def.~\ref{def:identity-open-nat-trans})}\\
        & = \Res{\nr[G]}[\Res{F}(k,x), y]] \cdot \Res{\theta}[k,\langle x, \langle \star, y \rangle \rangle] \\
        & = \id{\Res{G}(\Res{F}(k,x), y)} \cdot \Res{\theta}[k,\langle x, \langle \star, y \rangle \rangle] \text{ (by Def.~\ref{def:right-unitor})}\\
        & = \Res{\theta}[k,\langle x, \langle \star, y \rangle \rangle]
    \end{align*}
    \begin{align*}
        & \Res{\na[F,\Id{\C},G]}[k,\langle x, \langle \star, y \rangle \rangle] \\
        & = \id{\Res{G}(\Res{\Id{\C}}(\Res{F}(k,x),\star),y)} \text{ (by Def.~\ref{def:associator})}\\
        & = \id{\Res{G}(\Res{F}(k,x),y)}\text{ (by Def.~\ref{def:open-identity-functor})}
    \end{align*}
    Now the one step path.
    \begin{align*}
        & \Res{(\ID{G} \circ \nl[F])}[k,\langle x, y \rangle] \\
        & = \Res{\ID{G}}[\Res{\nl[F]}[k,x],y] \text{ (by Def.~\ref{def:open-horizontal-composition})}\\
        & = \Res{\ID{G}}[\id{\Res{F}(k,x)},y] \text{ (by Def.~\ref{def:left-unitor})}\\
        & = \Res{\ID{G}}[\Res{F}(k,x),y] \\
        & = \id{\Res{G}(\Res{F}(k,x),y)} \text{ (by Def.~\ref{def:identity-open-nat-trans})}
    \end{align*}
\end{proof}

\section{Conclusion}

This report presents all the formal details of the proof that the open functors, described as directly as possible, form a bicategory.
The subsequent reports will establish that
(1) this bicategory can be presented as the Kleisli bicategory of some pseudo-monad on $\Cat$,
(2) this pseudo-monad arise from a particularly simple adjunction linking this bicategory with $\Cat$,
(3) the formal definition in this report can be made easier to manipulate by the use of discrete fibrations instead of categories of elements through the so called Grothendieck construction, and doing so presents this bicategory as a particular bicategory of spans,
(4) this bicategory is a sub-bicategory of the bicategory of profunctors (\emph{a.k.a.} distributors),
(5) many manipulations in this report can be modularized because there is a factorization system separated the $\Alt{(\--)}$ part and the $\Res{(\--)}$ part.
(6) Kan extensions in this bicategory have a particular form in terms of the underlying strict 2-category $\Cat$.
Note that the many presentations of this bicategory is strongly related to the many possible presentations one can do on the notion of ``relation'': powerset monad for (1) and (2) above, spans for (3), and characteristic functions of the relation for (4).
These results are those that are already drafted and only need to be cleaned up and made available.

\bibliography{main}

\end{document}